\documentclass[a4paper, twoside, 11pt]{amsart}
\DeclareMathAlphabet{\mathpzc}{OT1}{pzc}{m}{it}
\usepackage[margin=3cm]{geometry}                 
\usepackage{graphicx}
\usepackage{amssymb}
\usepackage{mathrsfs}
\usepackage{enumitem}
\usepackage{array,multirow}
\usepackage{mathtools}
\usepackage{dsfont}
\usepackage{calligra}
\DeclareMathAlphabet{\mathcalligra}{T1}{calligra}{m}{n}
\usepackage{tikz}
\usepackage{tikz-cd}

\usepackage{float}
\usepackage{multicol}
\usepackage{rotating}
\usepackage{arydshln}
\usetikzlibrary{positioning}
\usepackage{caption}
\usepackage{subcaption}
\usepackage{mathdots}

\usepackage{ytableau}
\usepackage{hyperref}
\usepackage{amssymb,xspace,amscd,yhmath,mathdots,txfonts,
mathrsfs, youngtab,stmaryrd}
\usepackage{dynkin-diagrams}

\usepackage[matrix,arrow,curve]{xy}\CompileMatrices

\usepackage{yfonts}[1998/10/03]


\newtheorem{theorem}{Theorem}[section]
\newtheorem{lemma}[theorem]{Lemma}

\newtheorem{proposition}[theorem]{Proposition}

\newtheorem{remark}[theorem]{Remark}
\newtheorem{definition}[theorem]{Definition}

\newcommand\gt{\mathfrak{g}}
\newcommand\qt{\mathfrak{q}}

\newcommand\st{\mathfrak{s}}

\newcommand\sot{\mathfrak{so}}
\newcommand\slt{\mathfrak{sl}}

\newcommand\ft{\mathfrak{f}}
\newcommand\bt{\mathfrak{b}}
\newcommand\hg{\mathfrak{h}}

\newcommand\su{\mathfrak{su}}

\newcommand\C{\mathbb{C}}
\newcommand\R{\mathbb{R}}
\newcommand\Z{\mathbb{Z}}

\newcommand\K{\mathbb{K}}

\newcommand\sfM{\textsf{M}}
\newcommand\sfT{\textsf{T}}
\newcommand\sfF{\textsf{F}}
\newcommand\sfV{\textsf{V}}

\newcommand\sfa{\textsf{a}}
\newcommand\pct{\textsf{p}}
\newcommand\Ta{\textsf{H}}

\newcommand\qq{\mathpzc{q}}
\newcommand\vq{\mathpzc{v}}
\newcommand\wq{\mathpzc{w}}

\newcommand\Gf{\mathbf{G}}
\newcommand\Sf{\mathbf{S}}
\newcommand\Qf{\mathbf{Q}}

\newcommand\SU{\mathbf{SU}}

\usepackage{todonotes} 
\title{On Homogeneous CR Manifolds of Arbitrary Order of Levi Nondegeneracy }

\author{Stefano Marini}
\address{Stefano Marini, Department of Mathematical, Physical, and Computer Sciences, University of Parma, Parco Area delle Scienze 53/A (Campus), 43124 Parma, Italy }
\email{stefano.marini@unipr.it}

\author{Costantino Medori}
\address{Costantino Medori, Department of Mathematical, Physical, and Computer Sciences, University of Parma, Parco Area delle Scienze 53/A (Campus), 43124 Parma, Italy }
\email{costantino.medori@unipr.it}

\makeatletter
\@namedef{subjclassname@2020}{\textup{2020} Mathematics Subject Classification}

\makeatother

\thanks{The Authors were partially  supported by  PRIN 2022MWPMAB — “Interactions between Geometric Structures and Function Theories”
CUP: H53D23002040006; and and by 
the group G.N.S.A.G.A. of I.N.d.A.M.
}
\date{\today}
\begin{document}
\begin{abstract}

This paper present  homogeneous CR hypersurfaces satisfying the $CR$-invariant property of being $k$-nondegenerate for an arbitrary integer $k\geq 1$. The construction of such homogeneous manifolds is based on $CR$ algebras defined by irreducible representations of   $\su(2)$. An explicit study of the iterated Levi forms with their respective kernels, along with the local model equation, is given.
\end{abstract}
\maketitle

%\tableofcontents

\section*{Introduction}

In \emph{Cauchy-Riemann geometry}, a well-known problem is finding equivalence classes under the relation of CR diffeomorphisms. One approach to solving this problem is to find sufficiently many $CR$ invariants that characterize each class, called the equivalence problem. This problem has been classically solved for CR manifolds that are Levi nondegenerate real hypersurfaces and a model for the homogeneous case is given by
$$ 
 \SU{(1+p, q+1)}/   \mathbf{H} \simeq  \left\{\ 
[z_0,\ldots,z_{n+1}]
\,\vert \, \sum_{i = 0}^{p} |z^i|^2{=}\sum_{i = p+1}^{n+1} |z^i|^2  \ \right\} 
\subset \mathbb{CP}^{n+1}$$ (see  \cite{chern1974,Tanaka67,Tanaka70, Tanaka76}).  In particular, in the context of (locally) homogeneous manifolds, this problem can be studied using the tool of \emph{CR} algebras \cite{MN05}, i.e., a pair $(\mathfrak{g}^{\tau}, \mathfrak{f})$, where $\mathfrak{g}^{\tau}$ is a real Lie algebra encoding the group of CR diffeomorphisms, and $\mathfrak{f} \subset \mathfrak{g} = \mathbb{C} \otimes \mathfrak{g}^{\tau}$ is a complex Lie subalgebra encoding the CR structure. In \cite{MN98}, a classification of all semisimple 
Levi-Tanaka algebras was obtained. These are the 
Levi nondegenerate $CR$ algebras $(\mathfrak{g}^{\tau}, \mathfrak{f})$, 
where $\mathfrak{g}$ is semisimple, and $\mathfrak{g}^{\tau}$ inherits 
a $\mathbb{Z}$-gradation from $\mathfrak{g}$ that is compatible 
with the $CR$ structure. Several examples of 
homogeneous $CR$ manifolds are associated with 
$CR$ algebras $(\mathfrak{g}^{\tau}, \mathfrak{f})$ that  are $k-$nondegenerate for some $k{>} 1$, i.e., Levi degenerate but also satisfying a weaker condition of nondegeneracy according to their Levi form (see \cite[Def.11.1.8 p.317]{BER61999}).
In the case of a point $p$ in a generic analytic $CR$ hypersurface, this notion is equivalent to holomorphic nondegeneracy at $p$, that is the absence of a germ at $p$ of a holomorphic vector field $X$ tangent to $M$ (i.e., with a representative tangent to $M$ in a neighborhood of $p$) such that $X \mid_M \nequiv 0$ (see \cite[Th. 11.5.1, p. 330]{BER61999}). This, in turn, implies the finiteness of the dimension of the $CR$ automorphism group at $p$.
The $k-$nondegeneracy is widely used in the literature as a key invariant in solving the equivalence problem or studying  the group of $CR$ automorphisms, see \cite{IsaevZaitiev2013,MedoriSpiro2014,KruSan2025, ZelenkoSykes2023,MerkerPocchiola2020,KolarKoss2022,Greg2021}. 
In particular \cite{Fels2007,FelsKaup2008} the Authors pointed that it 
would be interesting to construct homogeneous $k-$nondegenerate $CR$ manifolds for arbitrarily large $k$. Quoting verbatim: ``We do not know how to construct $k$-nondegenerate homogeneous manifolds with $k \geq 5$ (should these exist)". In this paper  we address this question.\par  In Section~$1$, we recall some basic definitions in $CR$ geometry. In Section~$2$, we translate these definitions into the context of homogeneous $CR$ manifolds; in particular, we review some tools defined for non-semisimple homogeneous $CR$ manifolds. Finally, in Section~$3$, we define and characterize an example of a homogeneous $CR$ manifold that is $k$-nondegenerate for arbitrary $k \geq 1$, providing a local equation in Section~$4$ (Theorem~\ref{loceq}).
This last result should be compared, for $k = 1$, with \cite[Theorem~1.1]{Doubrov2021}, and for $k = 2$, with \cite[Table~1, Type~VII]{Sykes2023} (see also \cite[(19), p.~37]{Bel2023}). For arbitrary $k$, we also refer to \cite[Example~1]{Labovskii1997}.

\section{CR manifolds}
Let $\sfM$ be a $(2n+d)$-dimensional CR manifold, i.e. a differential manifold equipped with a \emph{CR structure} given by a rank $n$ complex distribution $\Ta^{0,1}\sfM\subset\C \otimes\sfT\sfM$ such that 
\begin{enumerate}
    \item $\Ta^{0,1}\sfM\cap\overline{\Ta^{0,1}\sfM}
    =\{0\}$
    \item $[\Gamma(\Ta^{0,1}\sfM),\Gamma(\Ta^{0,1}\sfM)]\subset\Gamma(\Ta^{0,1}\sfM)$.
\end{enumerate}
We call $(\sfM, \Ta^{0,1}\sfM)$ a \emph{Cauchy--Riemann manifold} of \emph{CR-dimension} $n$ and \emph{CR-codimension} $d$. When $d = 1$, we refer to it as being of \emph{hypersurface type}. A $CR$ manifold can also be given 
by a rank $2n$ real subbundle  {$\Ta_\R\sfM$} of $\sfT\sfM$ which is the{\emph{ real distribution}}
underlying the  CR  structure, i.e.,   $\Ta_\R\sfM=\textrm{Re}\left(\Ta^{0,1}\sfM\oplus\Ta^{1,0}\sfM\right)$ and a smooth complex structure  {$\textrm{J}_{\sfM}$}  on the fibers of $\Ta_{\R}\sfM$, defined by $\Ta^{0,1}\sfM{=}\{X+i\,\textrm{J}_{\sfM} {X}\mid X\in\Gamma(\Ta_\R\sfM)\},$ 
%The map $\textrm{J}_{\sfM}$  
that it squares to minus identity and  satisfies $\textsf{N}_{\textrm{J}_{\sfM}}(X,Y){=}[\textrm{J}_{\sfM}X,\textrm{J}_{\sfM}Y]-[X,Y]-\textrm{J}_{\sfM}([\textrm{J}_{\sfM}X,Y]+[X,\textrm{J}_{\sfM}Y])=0$ for any two sections $X,Y\in\Gamma(\Ta_\R\sfM)$,
taking the name of {\emph{partial complex structure}}
of~$\sfM$.\par
We label by $\mathcal{L}^{1}$ its \emph{Levi form}, which it is a sesquilinear 
form on $\Ta^{0,1}\sfM$ valued in 
$\mathbb{C}\sfT\sfM/(\Ta^{0,1}\sfM\oplus\overline{\Ta^{0,1}\sfM)}$ given at point 
$\pct$ in $\sfM $ by 
\begin{equation}\label{Levi form}
    \mathcal{L}^{1}(X_{\pct},Y_{\pct})=
\frac{1}{2i}\left[X,\overline{Y}\right]_{\pct}
\pmod{\Ta^{0,1}_{\pct}\sfM\oplus
\overline{\Ta^{0,1}_{\pct}\sfM}}
\end{equation}
for all $X,Y\in\Gamma(\Ta^{0,1}\sfM)$.
The kernel at ${\pct}\in \sfM$ of $\mathcal{L}^{1}$ is given by 
\begin{equation*}
    \mathcal{F}_{\pct}^1=\left\{v\in \Ta^{0,1}_{\pct}\sfM\,\left|\,\mathcal{L}^1(v,w)=0\,\, \forall\, w\in \Ta^{0,1}_{\pct}\sfM\right.\right\}.
\end{equation*}
When $\mathcal{F}_{\pct}^1$ is of constant rank at every point then  $\mathcal{F}^1=\bigcup_{\pct}\mathcal{F}_{\pct}^1$ is a vector distribution named \emph{Levi Kernel}.\par
A natural generalization of \eqref{Levi form} then is  the \emph{$k-$ order Levi form} 

\begin{equation}\label{K Levi forms}
    \mathcal{L}^k:\Ta^{0,1}\sfM
    \times \left(\bigodot^{k}\Ta^{0,1}\sfM\right)\to 
\mathbb{C}\sfT\sfM/(\Ta^{0,1}\sfM\oplus\overline{\Ta^{0,1}\sfM
})
    \end{equation}
  inductively defined at a point ${\pct}$ of $\sfM$ as 
\begin{equation*}
\mathcal{L}^{k}\left(X_{\pct},Y^1_{\pct},\ldots,Y^k_{\pct}\right)=\frac{
1}
{2i}\left[\cdots\left[\left[X,\overline{Y^1}\right],\cdots
\right],\overline{Y^k}\right]_{\pct}\pmod{\Ta^{0,1}_{\pct}\sfM\oplus\overline{\Ta^{0,1}_{\pct}\sfM}}
    \end{equation*}
with $X,Y^1,\ldots,Y^k\in\Gamma \left(\Ta^{0,1}\sfM\right)$.

 The left kernels of \eqref{K Levi forms} at $\pct\in\sfM$ is given by
\begin{align}\label{K kernel}
\mathcal{F}_{\pct}^k=
\{ 
v\in \Ta^{0,1}_{\pct}\sfM\,|\,\mathcal{L}^k(v,w_1,\ldots,w_k)=0\,\, \forall\, w_1,\ldots,w_k\in \Ta^{0,1}_{\pct}\sfM 
\}.
\end{align}
Again when the $\mathcal{F}_{\pct}$ are of constant rank at every points the vector distributions $\mathcal{F}^k=\bigcup_{{\pct}\in M}\mathcal{F}_{\pct}^k $ give rise to a nested sequence of complex subbundle 
of $\Ta^{0,1}\sfM$ (see \cite{Freeman1977}) 
\begin{equation}
\Ta^{0,1}\sfM\supseteq\mathcal{F}^1\supseteq\mathcal{F}^2\dots\supseteq\mathcal{F}^{k-1}\supseteq\mathcal{F}^{k}\supseteq\dots%\supseteq\{0\}. 
\end{equation}
\begin{definition}[Def.11.1.8\cite{BER61999}]\label{def knndg}
A CR manifold $\sfM$ is \emph{$k-$nondegenerate}, at ${\pct}\in\sfM$, if there exist a positive integer
$k$ such that
\begin{equation*}\label{k-nndg}
\mathcal{F}_{\pct}^{(k-1)} \supsetneqq \mathcal{F}_{\pct}^{(k)}{=}\{0\}, 
\end{equation*}
otherwise we say that it is \emph{finitely degenerate}.
\end{definition}
\section{Homogeneous CR manifolds and CR algebras}\label{CRalgebra}
From now on we will denote by $\gt^{\tau}$ a real form of a complex Lie algebra $\gt$, i.e., a subalgebra of $\gt$ such that $\gt^{\tau} \otimes \C = \gt$. To $\gt^{\tau}$, we can associate an anti-$\C$ linear involution $\tau$, that makes $\gt^{\tau}$ the subalgebra of the $\tau$-fixed elements of $\gt$.\par 

Let $\Gf^{\tau}$ be a Lie group with $\mathrm{Lie}\,\Gf^{\tau}=\gt^{\tau}$ acting transitively on a CR manifold $\sfM$. Fix a point $\pct \in \sfM$, and let
\[
\piup \colon \Gf^{\tau} \ni \sfa \mapsto \sfa \cdot \pct \in \sfM
\]
be the natural projection. The differential at $\pct$ defines a map
\[
\piup_{*} \colon \gt^{\tau} \to \sfT_{\pct}\sfM,
\]
which maps the Lie algebra $\gt^{\tau}$ of $\Gf^{\tau}$ onto the tangent space to $\sfM$ at $\pct$. By the formal integrability of the partial complex structure of $\sfM$, the complexification of the pullback of the $CR$ structure at $\pct$ is a complex Lie subalgebra
\[
\ft = (\piup^{*})^{\C}(\Ta^{0,1}_{\pct}\sfM)
\]
of the complexification $\gt$ of $\gt^{\tau}$. Conversely, the assignment of a complex Lie subalgebra $\ft$ of $\gt$ yields a formally integrable $\Gf^{\tau}$-equivariant partial complex structure on a $\Gf^{\tau}$-homogeneous space $\sfM$ by requiring that $\Ta_{\pct}^{0,1}\sfM = \piup^{\C}_{*}(\ft)$ (see \cite{MN97}).
These  lead to the following definition:
\begin{definition} A \emph{CR algebra} is a pair $(\gt^{\tau}, \ft)$, consisting of a real Lie algebra $\gt^{\tau}$ and a complex Lie subalgebra $\ft$ of its complexification $\gt,$ such that the quotient $\gt^{\tau}/(\gt^{\tau}\cap\ft)$ is a finite-dimensional real vector space. \end{definition} We call the intersection $\ft \cap \gt^{\tau}$ its \emph{isotropy subalgebra}. The $CR$-dimension and codimension of a $\Gf^{\tau}-$homogeneous $CR$ manifold $\sfM$ are expressed in terms of its associated CR algebra $(\gt^{\tau},\ft)$ by: \begin{equation*} \begin{cases} CR-\dim\sfM = \dim_{\C}\ft - \dim_{\C}(\ft{\cap}\tau(\ft)),\\ CR-\textrm{codim} \sfM = \dim_{\R}\gt - \dim_{\R}(\ft + \tau(\ft)). \end{cases} \end{equation*} \par

Now, given a CR algebra $(\gt^{\tau}, \ft)$, we have  
\begin{proposition}[Lemma 1.1,\cite{MN05}]
    The analytic tangent space  of the $CR$ homogeneous manifold with associated  $CR$ algebra $(\gt^{\tau}, \ft)$ at $\pct$ carries a unique complex structure $\textrm{J}_\sfM$ in $\pct$, characterized by the property:
\begin{equation}\label{PCS}
    \textrm{J}(X+\gt^\tau\cap\ft)=Y+\gt^\tau\cap\ft \textit{ if and only if }X+\sqrt{-1}Y\in\ft
\end{equation}
for $X,Y\in(\ft+\tau(\ft))\cap\gt^{\tau})/\gt^\tau\cap\ft$.\qed
\end{proposition}

When $\sfM$ is a (locally) homogeneous manifold,  $1$-nondegeneracy, i.e., the standard notion of Levi nondegeneracy in the literature,  at the point $\pct$ can be computed in terms of its associated CR algebra  \begin{equation*} \forall{Z} \in \ft{\backslash}\tau(\ft)\,\exists\, Z' \in \tau(\ft) \text{ such that }[Z, Z'] \notin \ft + \tau(\ft). \end{equation*} This is equivalent to asking that the pullback of the Levi kernel at $\pct$ coincides with the isotropy algebra \begin{equation*}
\ft^{{1}} \coloneqq \{Z \in \ft \mid [Z, \tau(\ft)] \subset \ft + \tau(\ft)\} = \ft \cap \tau(\ft) = (\ft \cap \gt^\tau) \otimes \C. \end{equation*} 
Similary one can look at $k$-nondegenracy at $\pct$ as \begin{equation*} \forall{Z} \in \ft{\backslash}\tau(\ft)\,\exists\, Z'_1,\dots,Z'_r  \in \tau(\ft), r\leq k, \text{ such that }[Z,[ Z'_1,[\dots,Z'_r]\dots] \notin \ft + \tau(\ft). \end{equation*}
 To measure the degeneracy of the Levi form, it is convenient to consider the sequence (see e.g. \cite[(4.3)]{
NMSM}): \begin{equation}\label{Fremanseq} \ft = \ft^{0} \supseteq \ft^{1} \supseteq \dots \supseteq \ft^{k-1} \supseteq \ft^{k} \supseteq \ft^{{k+1}} \supseteq \cdots \end{equation} defined recursively by
$$\ft^{k+1}=\{Z\in\ft^{k}\mid [Z,\tau(\ft)]\subseteq\ft^{k}+\tau(\ft)\}\;\;\;\text{for $k>0.$}$$

We note that $\ft \cap \tau(\ft) {\subseteq}\ft^{k}$ for all integers $k{\geq}0$. Since by assumption $\ft / (\ft \cap \tau(\ft))$ is finite-dimensional, there is a smallest nonnegative integer $k$ such that $\ft^{k}{=}\ft^{k'}$ for all $k'{\geq}k.$ 
\begin{definition} \label{CRalg}
A CR algebra $(\gt^{\tau},\ft)$ is said to be \emph{$k$-nondegenerate} if there is a smallest non-negative integer $k$ such that $\ft^{{k-1}} \supsetneqq \ft^{{k}} {=} \ft {\cap} \tau(\ft).$ 
\end{definition} 
\begin{proposition}[Sec.1.3, \cite{MMN21}]
The elements $\ft^{k}$ of the sequence \eqref{Fremanseq} are Lie subalgebras of $\gt$. 
\end{proposition} 
\begin{proof} By definition, $\ft^{0}{=}\ft$ is a Lie algebra. If $Z_{1}, Z_{2} ,{\in},\ft^{1},$ then \begin{align*} [[Z_{1}, Z_{2}], \tau(\ft)] \subseteq [Z_{1}, [Z_{2}, \tau(\ft)]] + [Z_{2}, [Z_{1}, \tau(\ft)]] \subseteq [Z_{1}+Z_{2}, \ft+\tau(\ft)] \subseteq \ft + \tau(\ft) \end{align*} because $[Z_{i}, \ft] {\subset}\ft$, since $Z_{i}$ belongs to the subalgebra $\ft$, and $[Z_{i}, \tau(\ft)] {\subset} \ft{+}\tau(\ft)$ by the definition of $\ft^{1}$. \par Now, we argue by induction. Let $k \geq 1$ and assume that $\ft^{k}$ is a Lie subalgebra. Then, if $Z_{1}, Z_{2} \in \ft^{k+1},$ we obtain: \begin{align*} [[Z_{1}, Z_{2}], \tau(\ft)] \subseteq [Z_{1}, [Z_{2}, \tau(\ft)]] + [Z_{2}, [Z_{1}, \tau(\ft)]] \subseteq [Z_{1}+Z_{2}, \ft^{k}] \subseteq \ft^{k}, \end{align*} because $\ft^{k+1} {\subseteq} \ft^{k},$ showing that also $[Z_{1}, Z_{2}] \in \ft^{(k+1)}.$ This completes the proof. \end{proof}
\begin{remark}\label{weaknndg}
The notion of $k$-nondegeneracy $CR$ algebra $(\gt^{\tau},\ft)$
can be proven to be equivalent to that  of \emph{weak nondegeneracy}  in \cite{MN05} ,  i.e.,  for any complex Lie subalgebra $\ft '$ of~$\gt,$
\begin{equation*} 
 \ft\subseteq\ft '\subsetneq\ft +\tau(\ft) \;\Longrightarrow \ft '=\ft.
\end{equation*}
The equivalence is given setting the largest complex Lie subalgebra $\ft '$ of $\gt$ 
with
 $\ft\,{\subseteq}\,\ft '\,{\subsetneq}\,\ft+\bar{\ft}$ 
 by \begin{equation*}
 \ft'=\ft+\tau\left({{\bigcap}_{k{\geq}0}\ft^{k}}\right)
\end{equation*}
( see  \cite[sec.1.3-1.4]{MMN21} for more details).
\end{remark}

In a series of papers (see e.g. \cite{
AMN06, 
AMN08, 
AMN06b,
Fels2007,
MaNa1, MaNa2, 
Wolf69}), several Authors have considered homogeneous $CR$ manifolds $\sfM$ for the action of a real form  a complex semisimple Lie groups $\Sf$. A large class of homogeneous $CR$ manifolds  are provided by real orbits  in a generalized complex flag manifold $\Sf /\Qf$ for a real form $\Sf^{\tau}$ of $\Sf$. Their $CR$ algebras are made of pairs $(\st^{\tau},\qt)$ at certain point, consisting of a real semisimple Lie algebra $\st^{\tau}$ and a parabolic Lie subalgebra $\qt$ of its complexification $\st$.
For a maximal parabolic $\Qf\subset\Sf$, i.e., not contained in others parabolic subgroup, G.Fels showed the follwing results:
%\SM{definisci chi è sigma}
\begin{theorem}[Th.4\,\cite{{Fels2007}}]
Take a complex semisimple Lie group $\Sf$ and a  parabolic subgroup $\Qf$. \par
Consider the complex flag manifold $\sfF=\Sf/\Qf$.
Fix a real form $\Sf^{\tau}$ of $\Sf$:
\begin{enumerate}
\item
Assume that $\Qf$ is \emph{maximal} parabolic. Then every non-open $\Sf^{\tau}$-orbit $\sfM \subset \sfF=\Sf/\Qf$ is $k$-nondegenerate.
In particular, if $\sfF=\Sf/\Qf$
is an irreducible
Hermitian symmetric space, then every non-open $\Sf^{\tau}$-orbit $\sfM\subset \sfF$ is $k$–nondegenerate with 
$k\leq 2$.
\item
Assume that $\Qf$ is not maximal parabolic.
Then there always exist non-open holomorphically degenerate $\Sf^{\tau}$-orbits in $\sfF$.
\end{enumerate}
\end{theorem}
In his paper, G.Fels raised  meaningful questions 
like what is the sharp bound for $k-$nondegenerate of $\Sf^{\tau}$-orbits. Additionally, he 
inquired how, in general, one might find examples 
of $k$-nondegenerate homogeneous CR manifolds for 
arbitrary values of $k$ greater than 3. We gave answers in the following forms:

\begin{theorem}[Th. 2.11\,\cite{MMN21}]
Let $\sfM$ be a $\Sf^{\tau}$-orbit in the complex flag manifold  $\sfF=\Sf/\Qf$. 
If $\sfM$ is $k$-nondegenerate, then   $k$ is  less or equal than $3$.
\end{theorem}
\begin{theorem}[Th.2.17\,\cite{MMN21}]
Let $\sfM$ the unique closed $\Sf^{\tau}$-orbit in the complex flag manifold  $\sfF=\Sf/\Qf$. 
If $\sfM$ is $k$-nondegenerate, then   $k$ is  less or equal than $2$.
\end{theorem}
We also provided a complete list of all closed real orbits that are $k$-nondegenerate, with $k=1,2$, as shown in \emph{Table I} and \emph{Table II} in \cite{MMN23}.
Moreover, G. Fels provided what is likely a first 
example of a homogeneous $3$-nondegenerate CR real 
orbits (see \cite[Sect.7]{Fels2007} ) and probably unique for long time example of a $4-$nondegenerate one  \cite[Example 5.1]{FelsKaup2008}. In \cite[Sect.3]{MMN21}, we exhibit a finite nondegenerate homogeneous $CR$ manifolds with Levi order $k$ for every positive integer $k$, by constructing certain vector bundles over $\mathbb{CP}^1$. \par 
For the above reasons, we need to move beyond the semisimple class of CR algebras if we want to construct 
$k$-nondegenerate homogeneous CR manifolds with $k\geq 3$.
 We will follow  \cite{AN2004}, where the Authors addressed the problem of describing those $CR$ vector bundles over compact 
% \SM{standard CR manifold esplicita o togli}
 (locally) homogeneous $CR$ manifolds that remain standard $CR$ manifolds, i.e., their associated $CR$ algebras are special graded  extensions of semisimple graded $CR$ algebras (see ~\cite[p.238]{HS}). Let $\st,\, \sfV$ be  Lie algebras over $\K.$  An \emph{extension of $\st$ by $\sfV$} is 
a Lie algebra $\gt$ over $\K$   which is the middle term of an exact sequence 
\begin{equation*} \label{e2.1}
\begin{CD}
 0 @>>> \sfV @>{\iotaup}>> \gt @>{\piup}>> \st @>>>0
\end{CD}\end{equation*}
of Lie algebra homomorphisms (see e.g. \cite[Ch.I, \S{1.7}]{Bour13}).  
If in  the  extension $\gt$ contains a Lie subalgebra  which is a linear complement of $\ker(\iotaup)$ 
and  if also it is  an ideal in $\gt$ then $\gt$ is
isomorphic to a semidirect product 
 $\st\oplus\sfV$. 
 \begin{remark}
     Let $\st$ be a Lie algebra over $\K$ and $\sfV$ an $\st$-module. Then   $\st\,{\oplus}\,\sfV$  admits a Lie algebra  
structure given by Lie product 
\begin{equation}\label{Liebracket}
\begin{cases}
    [X,Y]=[X,Y]_{\st}\\ 
    [X,\vq]=-[\vq,X]=\rho(X)\vq\\ 
    [\vq,\wq]=0
\end{cases}\,\,\,\,\,\,\,\
 \forall X,Y\in\st,\;
 \forall\vq,\wq\in\sfV.
\end{equation}
where $[\cdot,\cdot]_{\st}$ denotes the standard Lie bracket on $\st$ and $\rho$ a representation of $\st$ on $\sfV$. 

If a  $\Z$-gradation \begin{equation}\label{2.2}
\st\,{=}\,{\bigoplus}_{{\qq\in\Z}}\st_{\qq}\end{equation}
 of $\st$ is given, 
then a  \textit{compatible $\Z$-gradation} of an $\st$-module
$\sfV$ 
is a direct sum decomposition
\begin{equation}\label{2.3}
 V=\bigoplus_{p\in\mathbb{Z}} V_p \;\;\text{such that }\;\;
 [\st_p, V_q]\subset V_{p+q},\forall p,q\in\mathbb{Z}.
\end{equation}
For a $\Z$-graded Lie algebra \eqref{2.2}
the linear map %defined by 
\begin{equation}\label{gradingelement}
D:\st\ni X\rightarrow pX \in\st,\;\;\forall\,{p}\in\mathbb{Z},\;\forall\,{X}\in\st_{p},
\end{equation}
is a derivation. 
A derivation  $\mathcal{D}$ is said to be \emph{inner} if it exist an element $E\in\st$ such that $\mathcal{D}\,{=}\,ad_{E}$. Since all  derivations 
 of a finite dimensional semisimple Lie algebra are inner
and its adjoint representation is faithful,
 a semisimple $\Z$-graded Lie algebra \eqref{2.2} always admit a unique element $E,$ which belongs to the center of $\st,$ acting as \eqref{gradingelement}.  
\end{remark}
In the following sections, we will present a more in-depth study of this example, leading to an explicit model equation. 

 \section{A $k-$nondegenerate $CR$ manifold for $k\geq 1$}

Let $\st{=}
\slt_{2}
(\C)$ be the simple Lie algebra of $2{\times}2$ traceless matrices with complex coefficients, with the split real form $\mathfrak{sl}_{2}(\mathbb{R})$ given by the matrices with real coefficients.Then $\mathfrak{s}$ and its split real form admit a canonical basis over $\mathbb{C}$ and $\mathbb{R}$, respectively, given by

\begin{equation}\label{slbase}
H=\begin{pmatrix} 1&0\\ 0&-1 \end{pmatrix}, \quad
X^{\downarrow}= \begin{pmatrix} 0&0\\ 1&0 \end{pmatrix}, \quad
X^{\uparrow}= \begin{pmatrix} 0&1\\ 0&0 \end{pmatrix},  
\end{equation} 
which satisfy the bracket relations 
\begin{equation}\label{slbaserelation} [H,X^{\downarrow}] = -2X^{\downarrow},\,\,[H,X^{\uparrow}] = 2X^{\uparrow},\,\, [X^{\uparrow},X^{\downarrow}] = H. 
\end{equation} 
The characteristic element $E{=}\frac{1}{2}H$ yields the $\Z_2$-gradation of $\st$ \begin{equation}\label{A1_symm_space_grad} \st = \st_{-1} \oplus \st_0 \oplus \st_1, \text{ with }
\begin{cases} \st_{1} = \langle X^{\uparrow} \rangle_{\C},\\ \st_0 = \langle H \rangle_{\C},\\ \st_{-1} = \langle X^{\downarrow} \rangle_{\C}. 
\end{cases} 
\end{equation}

 Let us fix the Cartan 
and Borel subalgebras 
\begin{equation*}
 \hg=\left.\left\{ 
\begin{pmatrix}
 a & \; 0\\
 0 & {-}a
\end{pmatrix} \,\right| \,a\in\C\right\},\quad \bt=\left.
\left\{ \begin{pmatrix}
 a & \; b\\
 0 & {-}a
\end{pmatrix} \,\right| \,a,b\in\C\right\}
\end{equation*}
of $\slt_{2}(\C).$

Consider  $\mathbb{CP}^1$  as the Grassamanian  $\mathrm{F}=\mathbf{SL}_{2}(\C) / \mathbf{B}$ , of complex line through the origin inside $\C^2$, where $\mathbf{B}$ is the subgroup fixing $\langle e_1\rangle$. Since  the simply connected compact Lie group 
$$
{\SU}(2)=\left.\left\{\begin{pmatrix}
 a& b\\
 -\bar{b}& \bar{a}
 \end{pmatrix}\right| |a|^{2}{+}\,|b|^{2}=1\right\}$$
 acts transitively on $\mathrm{F}=\SU(2) / \mathbf{T}^{\tau}$, where  $\mathbf{T}^{\tau}$ is the Lie group generated by $\mathrm{exp}(iH)$,  its associated $CR$ algebra at the base point $e_1\in\C^2$ is given by $(\su(2),\bt)$, with $\mathrm{Lie}(\mathbf{B})=\bt=\st_{0}\oplus\st_{1}$.

\begin{remark}
Let us remember that $\slt_{2}(\R)$ is the split real form, i.e., the space of fixed loci of the conjugation $\sigma(X) = \bar{X}$ for every $X \in \slt_{2}(\C)$. The Lie algebra $\slt_{2}(\R)$ contains a maximal compact subalgebra given by the skew-symmetric matrix $\sot(2)$, i.e., the fixed loci of the Cartan involution $\theta(X) = -X^t$. The composition of $\sigma$ and the extension of $\theta$ to the complexification gives rise to the anti-$\C$ involution $\tau = \sigma \circ \theta^{\C}$, for which the set of fixed points gives us the compact real form of the special unitary matrix:
\[
\su(2) = \left\{ \begin{pmatrix} i t & b \\ {-}\bar{b} & -it \end{pmatrix} \;\middle|\; t \in \R, \; b \in \C \right\}.
\]
While $\eqref{slbase}$ still serves as a basis for the split real form over $\R$, for the compact form, we have the so-called \emph{Pauli matrices}:  
\begin{equation}\label{Pauli base}
\sigmaup_1{=}\frac{1}{2}(X^{\uparrow} {-}X^{\downarrow}) = {\frac{1}{2}} \begin{pmatrix} 0&1\\ -1&0 \end{pmatrix}, \quad
\sigmaup_2{=}\frac{i}{2}(X^{\uparrow} {+}X^{\downarrow}) = {\frac{i}{2}} \begin{pmatrix} 0&1\\ 1&0 \end{pmatrix}, \quad
\sigmaup_3{=}\frac{i}{2}H {=} {\frac{1}{2}} \begin{pmatrix} i&0\\ 0&-i \end{pmatrix}.
\end{equation}
Then the relations $\eqref{slbaserelation}$ induce the conditions on commutators:
\begin{equation}\label{commutators}
%[\sigmaup_i, \sigmaup_j] = \epsilon_{ijk} \sigmaup_k,
 [\sigmaup_1, \sigmaup_2] =\sigmaup_3,\quad [\sigmaup_3, \sigmaup_1] =  \sigmaup_2,\quad [\sigmaup_3, \sigmaup_2] = - \sigmaup_1.
\end{equation}
\end{remark}
The representations of $\SU(2)$ are used to describe non-relativistic spin, as $\SU(2)$ is a double cover of the rotation group $\mathbf{SO}(3)$ in the Euclidean 3-space. Relativistic spin, on the other hand, is characterized by the representation theory of $\mathbf{SL}_2(\C)$, which is the complexification of $\SU(2)$ and serves as a double cover of $\mathbf{SO}({1,3})$, the Lorentz group corresponding to relativistic rotations.\par
%\SM{chiamare w z  con z e bar z}
Now, consider the vector space ${\C}[z,\bar{z}]$ of complex polynomials in $z$ and $\bar{z}$, and take the $(2k{+}1)$-dimensional subspace $\sfV$ consisting of polynomials of degree $2k$:
\begin{equation*}
 \sfV=\left.\left\{{\sum}_{h=-k}^{k}a_{h}z^{k+h}\bar{z}^{k-h}\right| a_{h}\,{\in}\,\C\right\}.
\end{equation*} 
With the identification  
 $$\C^{2k+1}\ni(a_{-k},\dots,a_k)\longleftrightarrow {\sum}_{h=-k}^{k}a_{h}z^{k+h}\bar{z}^{k-h} $$ we have that the canonical basis over $\C$ is give by \begin{equation}\label{basepol}
   \{v_h=z^{k+h}\bar{z}^{k-h}\}_{ -k\leq h\leq k} . 
 \end{equation}
 
 For each $g\in\SU(2)$, define a linear transformation $\rho(g)$ on the space $\sfV$ by the formula

\begin{equation}\label{hompol} (\rho(g)\pct)(z,\bar{z})=\pct(g^{-1}(z,\bar{z})),\,\,\forall(z,\bar{z})\in\C^2. \end{equation}
The map $\rho$, as defined above, is a representation of $\sfV$. Note that $\SU(2)$ is isomorphic to the three-sphere $\mathbf{S}^3$, which is compact and simply connected. Therefore, the exponential map from its Lie algebra $\su(2)$ to the group $\SU(2)$ is surjective. As a result, the representations of $\SU(2)$ and $\su(2)$ are in one-to-one correspondence. Consequently, every representation of the Lie algebra can be studied as a group representation, which we will continue to denote by $\rho$. Moreover, every finite-dimensional representation $\rho$ of $\mathfrak{g}^\tau$ extends uniquely to a complex-linear representation of $\mathfrak{g}$, satisfying $\rho(X + iY) = \rho(X) + i\rho(Y), \quad \text{for all } X, Y \in \mathfrak{g}^{\tau}.
$
In particular, there is a one-to-one correspondence between the irreducible real representations of $\su(2)$ and the irreducible complex representations of $\slt_2(\mathbb{C})$.
Let $\rho$ be an irreducible representation of $\slt_{2}(\C)$ in a finite-dimensional complex vector space $\sfV$.
Consider the eigenvalues $\lambda$ and eigenvectors $v$ of $H$, such that $\rho(H)v{=}\lambda v$. We shall call $\lambda$ a weight of $\rho(H)$ and $v$ a weight vector, with $V_{\lambda}{=}\{v\in\sfV | \rho(H)v{=}\lambda v\}$ being the corresponding weight space. Using the notation $\rho(X)v{=}Xv$, we have:
 \begin{equation}\label{raising operator}
HX^{\uparrow}v=HX^{\uparrow}v+X^{\uparrow}Hv-
X^{\uparrow}Hv=
(\lambda+2)X^{\uparrow}v\Rightarrow X^{\uparrow}v\in\sfV_{
\lambda+2}
 \end{equation}
 and 
 \begin{equation}\label{lowering operator}
HX^{\downarrow}v=HX^{\downarrow}v+X^{\downarrow}v-X^{\downarrow}Hv=(\lambda-2)X^{\downarrow}v\Rightarrow X^{\downarrow}v\in\sfV_{\lambda-2},
 \end{equation}
 which is why $X^{\uparrow}$ is sometimes called the \emph{raising operator}, while $X^{\downarrow}$ is called the \emph{lowering operator}, as applying $X^{\uparrow}$ moves us “up” a weight space, and applying $X^{\downarrow}$ moves us “down” a weight space.
We remember that:
\begin{proposition}
\label{unique representation} For each integer $m{\geq}0$, there exists an irreducible complex representation of $\slt_{\C}(2)$ with dimension $m{+}1$. Any two irreducible complex representations of $\slt_{\C}(2)$ with the same dimension are isomorphic. If $\rho'$ is an irreducible complex representation of $\slt_{\C}(2)$ with dimension $m{+}1$, then $\rho'$ is isomorphic to the representation $\rho$ described in \eqref{hompol}.%\end{remark}
\end{proposition}
To better see the actions  \eqref{raising operator},\eqref{lowering operator} on $V_h$ consider 
the elements in \eqref{slbase}  identified by isomorphism  with \begin{equation} \label{isoderivatepart}
 H = z\partial_z - \bar{z}\partial_{\bar{z}},\ X^{\uparrow} = z\partial_{\bar{z}},\ X^{\downarrow} =  \bar{z}\partial_z
 \end{equation} since they satisfy the canonical relations \eqref{slbaserelation}: \begin{align*} [X^{\uparrow},X^{\downarrow}] &= z\partial_z - \bar{z}\partial_{\bar{z}} = H,\\ [H,X^{\uparrow}] &= z\partial_{\bar{z}} + z\partial_{\bar{z}} = 2X^{\uparrow},\\ [H,X^{\downarrow}] &= -\bar{z}\partial_z - \bar{z}\partial_z = -2X^{\downarrow}. \end{align*}
Fix $m{=}2k$, then by the action of the characterics element $E{=}\frac{1}{2}H$  on $\sfV$ by $\rho$ we have the $\Z$-gradation \begin{equation}\label{GradV}
    \sfV{=}\sum_{h=-k}
^{k}\sfV_{h}
\end{equation}
with $\sfV_{h}{=}\langle v_h\rangle_{\C}$
 and $v_h$ given by  formula \eqref{basepol}.
The unique extension of $\rho$, which we will still denote by $\rho$, from $\mathfrak{su}(2)$ to $\mathfrak{sl}_2(\mathbb{C})$ is given by $\rho(X + iY) {=} \rho(X) {+} i \rho(Y)$.  Then, using the uniqueness of the representation and the isomorphism \eqref{isoderivatepart}, we conclude that the extension of the conjugation $\tau$ on $\mathsf{V}$, which we still denote by $\tau$, involutively switches $z$ with $\bar{z}$,  sending the weight space $\langle v_h \rangle_{\C}$ to $\langle v_{-h} \rangle_{\C}$. Now consider on $\st$ the Borel subalgebra of upper  triangular matrix $\mathfrak{b}=\langle H,X^{\uparrow}\rangle_{\C}$
and define the complex subalgebra of lower triangular 
matrix given by the opposite Borel algebra $\mathfrak{b}^{opp}=\tau 
(\mathfrak{b}){=}\langle H,X^{\downarrow}\rangle_{\C}$. 
For $k\geq 1$ define the  semidirect product given by $\gt=\slt_{2}(\C)\oplus \sfV$ which by \eqref{A1_symm_space_grad} and \eqref{GradV} it is still a $\Z$ graded Lie algebra.  
The real form of $\gt$ is given by the fixed loci of $\tau$ 
$\gt^{\tau}=\su(2)\oplus \sfV^{\tau}$, where 
$\sfV^{\tau}$ is isomorphic to the unique irreducible real $\su(2)$-representation  of real dimension $2k+1$, which is the real form of 
 the unique one of $\slt_{\C}(2)$ of homogeneous polinomials in two variable of degree $2k$. % this time with base \eqref{basepol} but over $i\R$.
Define now the complex vector subspace 
\begin{equation}\label{struttura CR}
\ft=\mathfrak{b}\oplus \sfV^+\subset\gt
\end{equation}
with Lie bracket \eqref{Liebracket},  where $\sfV^+= \sum_{h>0} \sfV_h$ .
Then we obtain the following 

\begin{lemma}\label{computation}
Let $\ft$ be as in \eqref{struttura CR}. Then, we have:  

\begin{enumerate}
\item $\ft$ is a complex Lie subalgebra of $\gt$.
\item $\tau\ft$ is the subalgebra $\mathfrak{b}^{opp} \oplus \sfV^-$, where $\sfV^- = \sum_{h<0} V_h$.
%\SM{V++V- al posto di neq0}
\item $\ft + \tau(\ft) = \slt_{\C}(2) \oplus \sfV^+ \oplus \sfV^-$. 
\item $\ft \cap \mathfrak{g}^{\tau} = \langle iH \rangle_{\R}$, i.e., the isotropy subalgebra coincides with the maximally compact $\tau$-stable Cartan subalgebra $\mathfrak{t}^{\tau}$ of $\su(2)$ and $\ft \cap \tau\ft = \C \otimes (\ft \cap \mathfrak{g}_{\R}) = \mathfrak{b} \cap \mathfrak{b}^{opp}$, which is the Cartan subalgebra $\mathfrak{t} = \langle H \rangle_{\C}$ of $\slt_{\C}(2)$.
\item $(\gt^{\tau}, \ft)$ is a $k$-nondegenerate CR algebra of CR codimension $1$.
\end{enumerate}
\end{lemma}

\begin{proof}

We start by proving \((1)\), i.e., that \(\qt\) is a Lie subalgebra of \(\gt\).  
Since \([\sfV, \sfV] = 0\) and \(\st\) is a Lie algebra, we only need to verify this for the mixed Lie bracket, which holds because \([\mathfrak{b}, V^+] \subseteq V^+\). Statement \((2)\) follows by conjugating with \(\tau\). In fact, \(\tau{[\mathfrak{b}, \sfV^+]} = [\mathfrak{b}^{opp},{\sfV}^-] \subset \sfV^-\). From \((1)\) and \((2)\), it follows that \((3)\) and \((4)\) hold. For \((5)\), we explicitly construct the Fremann sequence \eqref{Fremanseq} using \((4)\) and the fact that \(\tau{\sfV_{h}} = \sfV_{-h}\). The computation can be visualized using the following diagram:  

\begin{equation}\label{tanbella}
\begin{ytableau}
v_{k} & v_{k{-}1} & \dots & v_{2} & v_{1} & v_{0} & v_{-1} & v_{-2} & \dots & v_{1{-}k} & v_{-k} \cr 
\setminus & \setminus & \setminus & \setminus & X^{\uparrow} & H & X^{\downarrow} & \setminus & \setminus & \setminus & \setminus 
\end{ytableau}
\end{equation}
which leads to  the Freeman's sequence

\begin{equation}\label{FS}
\begin{aligned}
\ft^{1} &= \{X \in \ft \mid [X, \ft] \subset \ft + \tau\ft\} = \langle H, v_{2}, \dots, v_{k} \rangle_{\mathbb{C}}, \\
\ft^{2} &= \{X \in \ft^1 \mid [X, \ft] \subset \ft^{1} + \tau\ft\} = \langle H, v_{3}, \dots, v_{k} \rangle_{\mathbb{C}}, \\
\vdots \\
\ft^{h} &= \{X \in \ft^{h-1} \mid [X, \ft] \subset \ft^{h-1} + \tau\ft\} = \langle H, v_{h+1}, \dots, v_{k} \rangle_{\mathbb{C}}, \\
\vdots \\
\ft^{k} &= \langle H \rangle_{\mathbb{C}} = \mathfrak{t}.
\end{aligned}
\end{equation}
This confirms that \((\gt^{\tau}, \ft)\) is a \(k\)-\textit{nondegenerate} CR algebra by definition.

\end{proof}

\begin{remark}
Every subalgebra $\ft'$ of $\gt$ containing $\ft$,  have to contain  $v_0$ by gradation, implying 
that  $(\gt^{\tau},\ft)$ is \emph{weak nondegenerate}, which is  equivalent to  $k-$ 
nondegeneracy, see  Remark \ref{weaknndg}.
\end{remark}

We notice that the  semidirect product $\gt^{\tau} {=} \mathfrak{su}(2) \oplus \sfV^{\tau}$ is a \textit{Levi-Malcev decomposition} (\cite[Sect.1.5]{MMN2020}), where $\sfV^{\tau}$ is the radical. We have indeed:

\begin{remark}\label{LinearForm}
Let $\sfV^{\tau}$ be the  $2k{+}1$ dimensional real irreducible representation  of $\su(2)$. Then the  extension $\gt^\tau=\su(2)\oplus\sfV^{\tau}$ can be representend by the matrices of the form

\begin{equation}\label{matrixform}
    \begin{pmatrix}
        {i2kt}&{-2kb}&&&&\vline &a_k\\
       \bar {b}&{i2(k{-}1)t}&\ddots&&&\vline&\\
       &\ddots&\ddots&&&\vline&\vdots\\
    %   &&\ddots&\ddots&\ddots&&\vline&\\
           &&&{{-}2i(k{-}1)t}&{{-}b}&\vline&\\
        &&&{2k\bar b}&{{-}2ikt}&\vline&a_{-k}\\\hline
        0&&\dots&&0&\vline&0\\
    \end{pmatrix}
\end{equation}
with $b \in \mathbb{C}$, $t \in \mathbb{R}$, $a_0 \in \mathbb{R}$, and $a_k \in \mathbb{C}$, for $k \in \{-k, \dots, -1, 1, \dots, k\}$, such that $a_{-j} = (-1)^j \bar{a_j}$. 

Indeed, the action of $\{H, X^{\uparrow}, X^{\downarrow}\}$ can be expressed in matrix form as a linear transformation (see \cite[p.26-27]{Samelson88}). 
%Similarly, we consider the action of $\gt^{\tau}$ on itself. 
Then the Pauli basis \eqref{Pauli base} is given by the matrices:
 $$\Sigma_1=\frac{1}{2}\begin{pmatrix}
    0&-2k&0&&&\dots&&&&0\\
    1&0&(-2k+1)&&&&&&&0\\
    0&2 &0&&&&&&&\\
    &&&\ddots&&&&&&\\
    &&&&0&-(k+1)&0&&&\\
    \vdots&&&&k&0&-k&&&\vdots\\
    &&&&0&(k+1)&0&&&\\
    &&&&&&&\ddots&&0\\ 
    &&&&&&&&0&-1\\
    0&&&&&\dots&&0&2k&0\\
\end{pmatrix}$$

 $$\Sigma_2=\frac{1}{2}\begin{pmatrix}
    0&2ki&0&&&\dots&&&&0\\
    i&0&(2k-1)i&&&&&&&0\\
    0&2i &0&&&&&&&\\
    &&&\ddots&&&&&&\\
    &&&&0&(k+1)i&0&&&\\
    \vdots&&&&ki&0&ki&&&\vdots\\
    &&&&0&(k+1)i&0&&&\\
    &&&&&&&\ddots&&0\\ 
    &&&&&&&&0&i\\
    0&&&&&\dots&&0&2ki&0\\
\end{pmatrix}$$
 and 
 $$\Sigma_3{=}\frac{i}{2}\textrm{diag}(k,k-1,\dots,1,0,-1,\dots,-k+1,-k).
 $$
The real span of $\Sigma_1, \Sigma_2, \Sigma_3$ gives rise to the upper-left matrix block of the matrices \eqref{matrixform} isomorphic to $\su(2)$,
while the last column represents $\sfV^{\tau}$. 
\end{remark}
The subalgebra $\ft \cap \gt^\tau $  is a $1$ dimensional toral Lie algebra with exponential
$$\mathbf{T}^{\tau}{=}{\left\{ 
\textrm{diag}(\exp (ikt),
\dots,\exp (it),1,\exp (-it),\dots,
,\exp (-kit),1)
 \,|\,
 t{\in}\R
\right\}}
$$
 isomorphic to the modulus one complex number and hence closed.  The homogeneous space $\sfM{=}\Gf^\tau/\mathbf{T}^{\tau}$, with $\textrm{Lie}(\Gf^\tau)=\gt^\tau$, is a well defined homogeneous manifold of real dimension dimension $2k{+}3$.
Let $\ft\subset\gt$  be as in \eqref{struttura CR}, 
then by Section \ref{CRalgebra} and 
the isomorphism $\textrm{T}\,\Gf^\tau/\mathbf{T}^{\tau}\simeq\Gf^{\tau}\times_{\mathbf{T}^{\tau}}\gt^{\tau}/\mathfrak{t}^{\tau}$ we have that 
\begin{equation}\label{CRS}
\Ta^{0,1}\sfM=\Gf^{\tau}\times_{\textbf{T}^{\tau}}(\ft/\mathfrak{t}),%\textit{ with }g\in\textbf{G}^{\R},
\end{equation}
define a $\Gf^\tau$-invariant $CR$ structure on the homogeneous manifold $\sfM=\textbf{G}^{\tau}/\textbf{T}^{\tau}$.
Since  $\textbf{G}^{\tau}$ is connected and $\textbf{T}^{\tau}$ is a closed subgroup having one connected components, using \cite[Th.A]{Most62} we obtain:
\begin{theorem}\label{thAMost}
Consider the $CR$ algebra $(\mathfrak{g}^\tau, \ft)$ given by Lemma \eqref{computation}. Let $\mathrm{SU}(2)$ be a maximal compact subgroup of $\mathbf{G}^{\tau}$ with Cartan subgroup $\mathbf{T}^{\tau} \subset \mathrm{SU}(2)$. Let $\sfV^{\tau}$ be the closed $(2k+1)$-dimensional irreducible $\mathrm{SU}(2)$-representation in $\mathbf{G}^{\tau}$ such that $  \mathbf{G}^{\tau} = \mathrm{SU}(2) \times \sfV^{\tau}$
as a topological direct product.
We define the homogeneous $CR$ manifold  
\begin{equation} \label{CRM}
    \mathsf{M} = \mathbf{G}^{\tau} / \mathbf{T}^{\tau} \simeq \mathrm{SU}(2) \times_{\mathbf{T}^{\tau}} \sfV^{\tau},
\end{equation}
where $(gk, v) \sim (g,\rho (k)v)$ for all $g \in \mathbf{G}^{\tau}, k \in \mathbf{T}^{\tau}, v \in \sfV^{\tau}$. Then the $CR$ algebra of Lemma \ref{computation} defines on $\mathsf{M}$ a $\mathfrak{g}^{\tau}$-invariant $CR$ structure $\Ta^{0,1} \mathsf{M} = \Gf^{\tau}\times_{\mathbf{T}^{\tau}}(\ft/\mathfrak{t})$, of CR dimension $k+1$ and CR codimension $1$, which is $k$-nondegenerate. This provides a fibration with  Euclidean fiber over $\mathbb{CP}^1$. \qed

\end{theorem}

    The elements of the   sequence \eqref{FS},  encoding the iterated Levi kernels, are given by 
\begin{equation}
    \mathcal{F}^{h>0}=\Gf^{\tau}\times_{\textbf{T}^{\tau}}(\ft^h/\mathfrak{t}),
\end{equation}
with elements of  $\ft^{h> 0}$ having linear form  by Remark \ref{LinearForm}

\begin{equation*}
    \left\{\left(
\begin{array}{c c c c c c c c c |c}
 kt&&&&&&&&&a_k\\
        &\ddots&&&&&&&&\vdots\\
        &&ht&&&&&&&a_h\\
         &&&0&&&&&&0\\
         &&&&\ddots&&&&&\vdots\\
         &&&&&0&&&&0\\
          &&&&&&-ht&&&a_{-h}\\
         &&&&&&&\ddots&&\vdots\\
         &&&&&&&&kt&a_{-k}\\
         \hline
          0&&&&\dots&&&&&0
\end{array}
\right)\,|\,t,a_{-k},\dots,a_{h-},a_{h},\dots, a_{k}\in\C\right\}.
\end{equation*}

By \eqref{PCS}, we can obtain a partial complex structure
$\textrm{J}_{\sfM} \in \mathrm{End}(\Ta_\R \sfM)$, constructed from the CR algebra $(\gt^{\tau},\ft)$ given in Lemma \ref{computation}, as follows. Consider the totally complex $CR$ algebra $(\su(2), \mathfrak{b})$ of $\mathbb{CP}^1$. Then, at the point $\pct = e_{\SU(2)}\mathbf{T}^\tau$, the complex structure $\textrm{J}_{\mathbb{CP}^1}$ is given by the adjoint action of the third element of the Pauli basis, i.e., $2\sigma_3 = iH$, on $\su(2)/\mathfrak{t}^\tau$.
Consider the action given by the representation $\rho$ of $\sigma_3$ on $\sfV^\tau$. Due to this, we can define a non-internal partial complex structure  $\textrm{J}_{\sfM}$ at $\pct' = e_{\mathbf{G}^\tau}\mathbf{T}^\tau \in \sfM$:

\begin{proposition}\label{partialcomplexstructure}
  Consider the $CR$ manifold $(\sfM, \Ta^{0,1} \sfM)$ given by \eqref{CRM} and \eqref{CRS}, respectively.
The linear map
\begin{equation}\label{CS}
\textrm{J}_{\sfM,\pct}(X) =
\frac{1}{h}\rho(\sigma_3)(X),  
\end{equation}
if $ X\in((\ft + \tau\ft) \cap \gt^\tau) / \mathfrak{t}^\tau $, where $h\neq 0$ is such that $X\in \langle v_h+\tau (v_{h})\rangle_{\R}$, defines the unique partial complex structure associated with the $CR$ algebra $(\gt^\tau, \ft)$ at
$\pct = e_{\Gf^\tau}\mathbf{T}^\tau$ on
$\Ta_\R\sfM_{\pct} {\simeq} = ((\ft + \tau\ft) \cap \gt^\tau) / \mathfrak{t}^\tau$. \qed

\end{proposition}

\begin{remark}\label{HOLF}

We noticed that the order of $k$-nondegeneracy of \eqref{CRM} at $\pct$ can be computed using the sequence \eqref{Fremanseq} associated with the $CR$ algebra $(\mathfrak{g}^\tau, \ft)$. In particular, the standard Levi form at the base $\pct$ is given by
\begin{equation*}
{\mathcal{L}}^{1}_{\pct}: \ft/\mathfrak{t} \times \tau \ft/\mathfrak{t} \ni (Z,\bar{W}) \longrightarrow [Z,\bar{W}] \textrm{ mod } \ft \oplus \tau \ft  \in \mathfrak{g}/(\ft  \oplus \tau \ft )
\end{equation*}
and it can be expressed in the basis $\{X^{\uparrow}, v_1, \dots, v_k\}$ by the $ (k{+}1) \times (k{+}1)$ Hermitian matrix
\begin{equation*}
    \begin{pmatrix}
    0 & 1 & \dots & 0 \\
    1 & 0 & & \vdots \\
    \vdots & & \ddots & \\
    0 & \dots&  & 0
\end{pmatrix}
\end{equation*}
which has constant \emph{rank} $2$ and mixed signature. Similarly, the $k$-order Levi form can be expressed at $\pct$ as
\begin{equation}\label{Lk}
{\mathcal{L}}^{h+1}_{\pct}: \ft^h/\mathfrak{t} \times \tau \ft/\mathfrak{t} \ni (Z,\bar{W}) \longrightarrow [Z,\bar{W}] \textrm{ mod } \ft^h \oplus \tau \ft\in (\ft^{h-1} \oplus \tau \ft)/(\ft^h \oplus \tau \ft).
\end{equation}

\end{remark}
\section{A Model for the $k-$nondegenerate CR hypersurfaces}
The manifold $\sfM$ of \eqref{CRM} is real analytic, and hence its $CR$ structure is locally induced by $\mathbb{C}^N$ for some $N \geq k+2$, with $\Ta^\R \sfM = \sfT\sfM \cap \sqrt{-1} \sfT\sfM$, where $\sqrt{-1}$ is the standard complex structure on $\mathbb{C}^N$. By uniqueness, the partial complex structure \eqref{CS} is the restriction of the standard multiplication by $\sqrt{-1}$.
In a neighborhood of the point $\pct \in \sfM$, the infinitesimal pseudo-conformal transfomation $X$, corresponding to $v_0$, is \emph{transversal}, i.e., $\sqrt{-1} X_{\pct} \notin \sfT_{\pct} \sfM$. The existence of such a section for $(\sfM, \Ta^{0,1} \sfM)$ coincides with the definition of being \emph{Tanaka regular at $\pct$} (see \cite[Definition 3, p.407]{Tanaka62}). Therefore we can find a system of complex coordinates $w, z_0, z_1, \dots, z_k$ at $p$ and a real-valued function $F(z_0,\dots, z_k,\bar z_0, \dots, \bar z_k)$ such that $\sfM$ is defined by the real-valued function
\begin{equation}\label{34} \textrm{Re}(w) = F(z, \bar{z}), \end{equation}
(see \cite[Prop.4, p.408]{Tanaka62}).\par
  Now consider the hypersurface $\mathsf{M}$ in the space $\mathbb{C}^{k+2}$, with coordinates $z=(w, z_0, \dots, z_k)$, given by the real equation  
\begin{equation}\label{eq0}  
    \text{Re}(w) = 2\sum_{h=1}^{k} \text{Re} (z_0^h \bar{z}_h).  
\end{equation}  
This example was  presented in \cite[Example 1]{Labovskii1997}
where it was shown that it is homogeneous and $\mathfrak{hol}(M)$, i.e. the set of vector fields $Z{=}b(z)\frac{\partial}{\partial w}+\sum_{h=0}^ka(z)\frac{\partial}{\partial z_h},$  with holomorphic coefficients such that $\textrm{Re}(Z)$  is tangent to $\sfM$, is finite dimensional (\cite[Example 1]{Labovskii1997}). We start observing that the Levi form of \eqref{eq0}, defined by $\frac{\partial^2 F}{\partial z \partial \bar{z}},$
at zero  coincides with the one builded with the $CR$ algebra  $(\su(2)\oplus\sfV^{\tau}, \ft)$ of Lemma \eqref{computation} at \( \pct \).  A similar observation applies to the higher-order Levi forms of the hypersurface \( \sfM \) in \eqref{eq0} and the maps in Remark \ref{HOLF}.

\begin{proposition}\label{abelianpart}
Fix the hypersurface $\sfM$ given by \eqref{eq0} and consider the complex vector fields 
\begin{equation}\label{Zh}
    {Z}_h = \frac{1}{2} \frac{\partial}{\partial z_h} + z_0^h \frac{\partial}{\partial w}, \quad  
    {Z'}_h = i \left( \frac{1}{2} \frac{\partial}{\partial z_h} - z_0^h \frac{\partial}{\partial w} \right), \quad\text{for }1\leq h\leq k,
\end{equation}
 
\begin{equation}\label{W}
    {W} = i \frac{\partial}{\partial w},
    \end{equation}
 \begin{equation}\label{Ahj}
    A_{hj}=\frac{1}{2}(z^h_0\frac{\partial}{\partial z_j}-z_0^j\frac{\partial}{\partial z_h}),\quad A'_{hj}=-\frac{i}{2}(z^h_0\frac{\partial}{\partial z_j}+z_0^j\frac{\partial}{\partial z_h}), 
\quad\text{for }1\leq h<j\leq k,
\end{equation}
and
\begin{equation}\label{Ah}
A'_h=iz_0^h\frac{\partial}{\partial z_h},
\quad\text{for }1\leq h\leq k.
\end{equation}
Then the vector fields given in \eqref{Zh}, \eqref{W}, \eqref{Ahj}, and \eqref{Ah} have real parts tangent to $\sfM$ and pairwise commute. They generate a $(k+1)^2$-dimensional abelian subalgebra of $\mathfrak{hol}(\sfM)$, which we denote by $\mathcal{W}$.
\qed
\end{proposition}
We observe that the vector fields in \eqref{Ahj} and \eqref{Ah}  vanish at the origin. 
We now define three complex vector fields, $\mathcal{J}$, $\mathcal{E}$, and $\mathcal{K}$. The first of them, when rescaled by $h$, will relate the primed and unprimed vector fields, as well as $\mathcal{E}$ and $\mathcal{K}$ will be associated with a grading element.

\begin{proposition}\label{cpxstrVF}
Fix the hypersurface $\sfM$ given by \eqref{eq0} and consider the complex vector fields that vanishes at the origin:
\begin{equation}\label{E}
    \mathcal{E} = (k{+}1)w \frac{\partial}{\partial w}+ z_0 \frac{\partial}{\partial z_0} + \sum_{h=1}^{k} (k{+}1{-}h) z_h \frac{\partial}{\partial z_h},
\end{equation}
\begin{equation}\label{J}
    \mathcal{J} = -i \left( z_0 \frac{\partial}{\partial z_0} + \sum_{h=1}^{k} h z_h \frac{\partial}{\partial z_h} \right),
\end{equation}

\begin{equation}
    \mathcal{K}=-z_0\frac{\partial}{\partial z_0}+\sum_{h=1}^khz_h\frac{\partial}{\partial z_h}.
\end{equation}
Then we have:
\begin{enumerate}
    \item $\mathcal{E}$, $\mathcal{J} \textit{ and } \mathcal{K} \in \mathfrak{hol}(\sfM)$;
     \item $[\mathcal{J}, \mathcal{E}]=[\mathcal{K}, \mathcal{E}]=[\mathcal{J}, \mathcal{K}]=0$;
    \item Considering the vector fields $Z_h, Z'_h$, and $W$ defined in \eqref{Zh} and \eqref{W}, we obtain:
   \begin{equation*}
       [\mathcal{J}, Z_h] = h Z'_h, \quad [\mathcal{J}, Z'_h] = -h Z_h, \quad \textit{for}\, 1\leq h\leq k,
   \end{equation*}
   \begin{equation*}
       [\mathcal{J}, W] = 0.
   \end{equation*}
     \item Considering the vector fields $A_{h,j}, A'_{h,j}$, and $A'_h$ defined in \eqref{Ahj} and \eqref{Ah}, we obtain:
   \begin{equation*}
       [\mathcal{J}, A_{h,j}] = (h-j) A'_{h,j}, \quad [\mathcal{J}, A'_{h,j}] = -(h-j) A_{h,j}, \quad \textit{for}\, 1\leq h<j\leq k,
   \end{equation*}
   \begin{equation*}
       [\mathcal{J}, A'_h] = 0 \quad \textit{for}\, 1\leq h\leq k.
   \end{equation*}
   \item The element $\mathcal{K}+\mathcal{E}$ act as $-(k+1)\mathrm{Id}$ on the abelian algebra $\mathcal{W}$ of Proposition \ref{Ah}.
   
\end{enumerate}
\qed
\end{proposition}
Moreover:

\begin{proposition}\label{AscDesVF}
Fix the hypersurface $\sfM$ given by \eqref{eq0} and consider the complex vector fields: 
\begin{equation}
   Z_- = k\left(\frac{\partial}{\partial z_0} + 2z_1 \frac{\partial}{\partial w} - \sum_{h=1}^{k-1} (h+1) z_{h+1} \frac{\partial}{\partial z_h}\right),
\end{equation}
\begin{equation}
   Z'_- = i k\left( \frac{\partial}{\partial z_0} - 2z_1 \frac{\partial}{\partial w} + \sum_{h=1}^{k-1} (h+1) z_{h+1} \frac{\partial}{\partial z_h} \right),
\end{equation}
\begin{equation}
   Z_+ =  \frac{1}{k} z_0^2 \frac{\partial}{\partial z_0} + z_0 w \frac{\partial}{\partial w} + \left( z_1 z_0 - \frac{1}{2} w \right) \frac{\partial}{\partial z_1} + \sum_{h=1}^{k-1} \left( z_{h+1} z_0 + \left(1 - \frac{h}{k} \right) z_h \right) \frac{\partial}{\partial z_{h+1}},
\end{equation}
\begin{equation}
   Z'_+ = -\frac{i}{k} z_0^2 \frac{\partial}{\partial z_0} - iz_0 w \frac{\partial}{\partial w}+ \left( -iz_1 z_0 - \frac{i}{2} w \right) \frac{\partial}{\partial z_1}  + \sum_{h=1}^{k-1} \left( -iz_{h+1} z_0 + i\left(1 - \frac{h}{k} \right) z_h \right) \frac{\partial}{\partial z_{h+1}}.
\end{equation}
Then we have:
\begin{enumerate}
    \item $Z_-, Z'_-, Z_+, Z'_+ \in \mathfrak{hol}(\sfM)$;
    \item $[Z_-, Z'_-] = [Z_+, Z'_+] = 0$;
    \item $[\mathcal{J}, Z_-] = Z'_-$ and $[\mathcal{J}, Z'_-] = -Z_-$;
    \item $[\mathcal{J}, Z_+] = Z'_+$ and $[\mathcal{J}, Z'_+] = -Z_+$;
    \item $[Z_+,Z'_-] = \frac{1}{2} \mathcal{J}$  and $[Z'_+, Z_-] = -\frac{1}{2} \mathcal{J}$;
       \item $[\mathcal{K}+\mathcal{E},Z_+]=[\mathcal{K}+\mathcal{E},Z_-]=[\mathcal{K}+\mathcal{E},Z'_+]=[\mathcal{K}+\mathcal{E},Z'_-]=0$.
  
 \end{enumerate}
\qed
\end{proposition}

\begin{lemma}\label{slcopia}
Fix the hypersurface $\sfM$ given by \eqref{eq0} and consider the complex vector fields given in Propositions \ref{cpxstrVF} and \ref{AscDesVF}. Set  
\begin{equation}
     \mathcal{H}{:=}[Z_+,{-}Z_-]{=}[Z'_+,{-}Z'_-]{=}\frac{2}{k{+}1}\left(k\mathcal{E}{-}\mathcal{K}\right){=}
    2\left(kw\frac{\partial}{\partial w}{+}z_0\frac{\partial}{\partial z_0}+\sum_{h=1}^k(k{-}h)z_h\frac{\partial}{\partial z_h}\right).
\end{equation}
The vector fields $Z_+, Z_-$, and $\mathcal{H}$ generate a Lie algebra isomorphic to $\mathfrak{sl}_2(\C)$ inside $\mathfrak{hol}(\sfM)$. Moreover, the element $\mathcal{H}$ gives rise to a  $2\Z$-gradation, that ranges from $-2k$ to $2k$, of the elements of $\mathcal{W}$ of Proposition \ref{abelianpart} as follow: 
\begin{equation*}
    [\mathcal{H},Z_h]=-2(k-h)Z_h,\,\quad
    [\mathcal{H},Z'_h]=-2(k-h)Z'_h,\quad \textit{for}\, 1\leq h,\,\leq k\quad [\mathcal{H},W]=-2kW,  
\end{equation*}
\begin{equation*}
    [\mathcal{H},A'_{h,j}]=2((h+j)-k)A'_{h,j},\quad [\mathcal{H},A_{h,j}]=2((h+j)-k)A_{h,j}, \quad \textit{for}\,  1\leq h<j\leq k,
\end{equation*}
\begin{equation*}
    [\mathcal{H},A'_h]=2(2h-k)A'_h,\quad \textit{for}\,    1\leq h\leq k.
\end{equation*}
\end{lemma}
\begin{proof}
The vector fields $Z_+, -Z_-$, and $\mathcal{H}$ satisfy the relations in \eqref{slbaserelation}, and the isomorphism $ Z_+\mapsto X^{\uparrow}$, $ -Z_-\mapsto X^{\downarrow}$ and $ \mathcal{H}\mapsto H$, with the standard bases of $\mathfrak{sl}_2(\C)$, is well defined.
\end{proof}
We notice that by Proposition \ref{cpxstrVF} the element $\pm\frac{2}{k+1}(\mathcal{E}+\mathcal{K})$ can be used to move up or down the $2\Z$-gradation of $\mathcal{W}$. 

\begin{lemma}\label{sucopia}
Fix the hypersurface $\sfM$ given by \eqref{eq0} and consider the complex vector fields given in Propositions \ref{cpxstrVF} and \ref{AscDesVF}. Set  
\begin{equation}
    \tilde\Sigma_1 = \frac{1
    %\sqrt k
    }{2} \left(\mathpzc{Re}(Z_+)-\mathpzc{Re}(-Z_-)\right), \quad  
    \tilde\Sigma_2 = \frac{i}{2} \left(\mathpzc{Re}(Z_+)+\mathpzc{Re}(-Z_-)\right), \quad  
   \tilde\Sigma_3 = [\Sigma_1,\Sigma_2].
\end{equation}
Then the vector fields $\Sigma_1, \Sigma_2$, and $\Sigma_3$ generate a Lie algebra isomorphic to $\mathfrak{su}(2)$  tangent to $\sfM$. Moreover we have
\begin{equation}
    \tilde\Sigma_3=\frac{i}{2}\mathpzc{Re}(\mathcal{H}).
\end{equation}
\end{lemma}

\begin{proof}
By the relations in Proposition \ref{AscDesVF}, we conclude that $\Sigma_1, \Sigma_2$, and $\Sigma_3$ satisfy the relations in \eqref{commutators}, and the isomorphism $\Sigma_i \mapsto \sigma_i$, for $i \in \{1,2,3\}$, with the Pauli basis of $\mathfrak{su}(2)$, is well defined.
\end{proof}
We note that the element $\Sigma_3$ acts on the real parts of the elements described in Proposition~\ref{abelianpart}, as specified in Proposition~\ref{partialcomplexstructure}.
\begin{remark}
    We observe that the elements
    \begin{equation}
        \Sigma_1 = \frac{1}{2}(Z_+ + Z_-), \quad  
       \Sigma_2 = \frac{1}{2}(Z'_+ + Z'_-), \quad  
        \Sigma_3 = \mathcal{J}.
    \end{equation}
   generates a copy of the Lie algebra $\su(2)$  into $\mathfrak{hol}(\sfM)$.
    \end{remark}
 Combining all the results above, we obtain:

\begin{theorem}\label{loceq}
  Given a fixed natural number \( k \geq 1 \), the homogeneous $k$-nondegenerate CR hypersurface defined in Theorem \ref{thAMost}, with the CR algebra given in Lemma \ref{computation}, is locally described by the hypersurface $\sfM$ with equation:
  \begin{equation*}\label{eq}
        \textrm{Re}(w) = 2\sum^k_{h=1} \textrm{Re} (z_0^h \bar{z}_h),
\end{equation*}
where \( (w, z_0, \dots, z_k) \) are local coordinates in \( \mathbb{C}^{k+2} \).
\end{theorem}

\begin{proof}
We will construct a set of vector fields tangent to $\sfM$  that generates a Lie algebra isomorphic to \( \mathfrak{g}^\tau \) given in  Lemma \ref{computation}. Moreover we will define a CR structure compatible with the  one in Theorem \ref{thAMost}. By Lemma \ref{slcopia}  the vector fields $Z_+,Z_-$ and $\mathcal{H}$ generates a Lie algebra isomorphic to $\slt_2(\C)$ inside $\mathfrak{hol}(\sfM)$.
By Lemma \ref{sucopia}  the vector fields $\tilde\Sigma_1,\tilde\Sigma_2$ and $\tilde\Sigma_3$ generates a Lie algebra isomorphic to $\su(2)$ tangent $\sfM$. We have that \([Z_-, W] = [Z'_-, W] = 0\) and \([Z_+, A'_k] = [Z'_+, A'_k] = 0\). The space \(\mathcal{V}\), generated by \(\operatorname{ad}^j(Z_+)(W)\) for \(j \in \{0, \dots, 2k\}\), is a complex  \(2k+1\)-dimensional $\slt_2(\C)$-representation. Each \(\operatorname{ad}^j(Z_+)(W)\) spans a non-zero complex one-dimensional eigenspace of \(\mathcal{H}\) contained in the subspace generated by elements of grade $h=j-k$, with  \(-k \leq h \leq k\), consistent with the structure given in Lemma~\ref{slcopia}.
By the grading structure described in Lemma~\ref{slcopia}, the representation under consideration is finite-dimensional, $2\Z-$graded, and has the property that each eigenvalue occurs with multiplicity one and therefore it is irreducible \cite[Corollary 4.2.4]{Var84}. Moreover, there is a one-to-one correspondence between finite-dimensional, irreducible complex representations of $\mathfrak{su}(2)$ and those of $\mathfrak{sl}(2, \mathbb{C})$. In particular, the irreducible representations of $\mathfrak{su}(2)$ can be viewed as real forms of the irreducible representations of $\mathfrak{sl}(2, \mathbb{C})$. It follows that the complex representation $\mathcal{V}$ contains a real, irreducible representation of $\mathfrak{su}(2)$ of dimension $2k+1$, which we denote by $\mathcal{V}^\tau$. 
    The conjugation operator \(\tau\) maps the  vector space generated by $\operatorname{ad}^j(Z_-)(A'_k)$  to the conjugate generated by $\operatorname{ad}^j(Z_+)(W)$ , and through $\tau-$symmetrization, yields the corresponding real irreducible representation $\mathcal{V}^\tau$ of $\mathfrak{su}(2)$, which corresponds to the standard representation of $\mathfrak{sl}_2(\mathbb{C})$. As a consequence, the vector fields $\mathcal{H}$, $Z_-$, and the set $\{\operatorname{ad}^j(Z_-)(A'_k)\}$, with $0{\leq} j{\leq} k-1$, generate a complex Lie subalgebra isomorphic to the algebra $\ft$ described in Lemma~\ref{computation}. The $CR$ distribution and structure can be rebuild from this data. This completes the proof.
\end{proof}

\bibliographystyle{amsplain}
\renewcommand{\MR}[1]{}
\providecommand{\bysame}{\leavevmode\hbox to3em{\hrulefill}\thinspace}
\providecommand{\MR}{\relax\ifhmode\unskip\space\fi MR }

\end{document}